\documentclass[10pt,reqno, english]{amsart}  
\usepackage[utf8]{inputenc}
\usepackage[T1]{fontenc}
\usepackage{amsmath,amsthm}
\usepackage{amsfonts,amssymb,enumerate}
\usepackage{url}
\usepackage{mathtools}  
\usepackage{enumerate, paralist}
\usepackage{anysize}
\usepackage{units}
\usepackage{tikz}

\newtheorem{theorem}{Theorem}[section]
\newtheorem*{theorem*}{Theorem}
\newtheorem*{lemma*}{Lemma}
\newtheorem*{prop*}{Proposition}
\newtheorem{lemma}[theorem]{Lemma}
\newtheorem{cor}[theorem]{Corollary}
\newtheorem{prop}[theorem]{Proposition}

\theoremstyle{definition}

\newtheorem{defn}[theorem]{Definition}

\newcommand{\R}{\mathbb{R}}

\begin{document}

\title[Monotone Paths on Cross-Polytopes]{Monotone Paths on Cross-Polytopes}



\author[A.~Black \and J.A.~De Loera]{Alexander E. Black \and Jes\'{u}s A. De Loera}

\address[AB]{Dept.\ Math., UC Davis, Davis, CA 95616, USA}
\email{aeblack@ucdavis.edu} 

\address[JDL]{Dept.\ Math., UC Davis, Davis, CA 95616, USA}
\email{deloera@ucdavis.edu}


\begin{abstract}
In the early 1990's, Billera and Sturmfels introduced the monotone path polytope (MPP), a special case of the general theory of fiber polytopes that associates a polytope to a pair $(P,\varphi)$ of a polytope $P$ and linear functional $\varphi$. In that same paper, they showed that MPPs of simplices and hyper-cubes are combinatorial cubes and permutahedra respectively. Their work has lead to many developments in combinatorics.
Here we investigate the monotone paths for generic orientations of cross-polytopes. We show the face lattice of its MPP is isomorphic to the lattice of intervals in the sign poset from oriented matroid theory. We look at its $f$-vector, its realizations, and facets.
\end{abstract}

\maketitle


\section{Introduction} 
In their seminal paper \cite{BSFiberPoly}, Billera and Sturmfels developed a construction that, given a projection of polytopes, associates a new polytope to that projection called the \emph{fiber polytope} (see the books \cite{MYBOOK, zieg} for an introduction). Fiber polytopes have rich combinatorial structure as demonstrated by the fact that associahedra, permutahedra, cyclohedra, and other combinatorial polytopes are all  fiber polytopes of canonical projections \cite{holypaper,chapoton_fomin_zelevinsky_2002, MYBOOK, HOHLWEGetal, Baues, vic-equifiber}. Fiber polytopes are of great importance beyond algebraic and geometric combinatorics too (see for example the connections to algebraic geometry in \cite{GKZbook, Mcdonald, SturmfelsYu} and recently to theoretical physics through total positivity in \cite{nimapaper, highersecond}).

Fiber polytopes extract complicated combinatorial structure even from one-dimensional projections. The fiber polytopes of one-dimensional projections are called \emph{monotone path polytopes} (MPPs), since they are each the convex hull of all average values of monotone paths on the polytope. The MPPs of simplices for generic linear functionals are combinatorial hyper-cubes, and the MPPs of hyper-cubes for generic linear functionals are always permutahedra \cite{BSFiberPoly}. Few examples of MPPs or fiber polytopes in general beyond these special cases are known, which makes studying them difficult. In this note, we develop a new, natural class of examples in depth: the MPPs of cross-polytopes for generic orientation.

To compute the combinatorial type of monotone path polytopes, it suffices to understand the poset of cellular strings or \emph{Baues poset} (see \cite{Baues}). Cellular strings generalize monotone paths in the sense that monotone paths are increasing sequences of edges, and cellular strings are increasing sequences of faces. The face lattice of a MPP of $\diamond^{n}$ is isomorphic to the poset of \emph{coherent} cellular strings contained within the Baues poset. Coherence is a geometric restriction that we will make precise in Section \ref{sec:bg}; it intuitively means that there is a projection of the polytope to a polygon built using the linear functional that takes the cells in the string to the lower edges of that polygon. For any hyper-cube, simplex, and any edge generic orientation, all cellular strings are coherent as stated in \cite{BSFiberPoly}. The cellular strings of the simplex correspond to intervals $[A,B]$ in the poset of subsets of $[n-2]$, where $A$ is the set of endpoints of the string, and $B$ is the set of all vertices that appear anywhere in the string. However, for general polytopes, not all monotone paths are coherent, and the characterization of the coherent paths often leads to interesting combinatorics such as in the cases discussed in \cite{cubepiles, edmanthesis, Hypersimps}. 
 
 In this paper, we study the monotone paths and the cellular strings of the standard $n$-dimensional \emph{cross-polytope} $\diamond^{n}$ given by the convex hull of the $2n$ vectors $e_1,-e_1,e_2,-e_2,\dots,e_n,-e_n$. As a polyhedron, $\diamond^{n}$ is given by the $2^n$ inequalities of the form  $\pm x_1 \pm x_2 \dots \pm x_d \leq 1$ for all possible sign choices. Cross-polytopes are famous simplicial polytopes polar to cubes, and a subset of vertices of the cross-polytope is a face if and only if it does not contain pairs of antipodes. In what follows, $P^{\Delta}$ denotes the polar dual of a polytope $P$. With these facts in mind, we may state our main result:
 
\begin{theorem}
\label{MainTheorem1} For the standard cross-polytope $\diamond^{n}$ and for any generic linear functional $\varphi$ such that $\varphi(e_{i}) = a_{i}$ for all $i \in [n]$, 
\begin{enumerate}

\item[(a)] If $0 < a_{1} < a_{2} < \dots < a_{n}$, the set of vertices of $MPP_{\varphi}(\diamond^{n})$ is precisely:
\begin{align*}
   &\Bigg{\{} \left(1 - \frac{a_{i_{k}} + a_{i_{1}}}{2a_{n}}\right)e_{n} + \sum_{i = 1}^{k} \left(\frac{a_{i_{k-1}} + a_{i_{k+1}}}{2a_{n}}\right) e_{i_{k}}: \\
    &-n = i_{0} < \dots < i_{k+1} = n \text{ and } i_{a} \neq -i_{b} \text{ for all } a,b \in [k]\Bigg{\}}.
\end{align*}

\item[(b)] There is an explicit polyhedral realization of $MPP_{\varphi}(\diamond^{n})$. If $0< a_{1} < a_{2} < \dots < a_{n}$, then $MPP_{\varphi}(\diamond^{n})$ is given by
    \[\{x \in \R^{n}: \varphi(x) = 0 \text{ and } \varphi_{i, \varepsilon}(x) \geq -a_{i} - a_{n}, \varepsilon: [n-1] \to \{\pm 1\}, k \in [n-1]\},\]
where we define $\varphi_{i, \varepsilon}$ on the basis $F_{1} \cup F_{2} \cup \{e_{n}\}$ by
\[\varphi_{i, \varepsilon}(e_{k}) = \begin{cases}  -a_{k} - a_{n} \text{ if } k \in F_{1} \\   \frac{a_{i} + a_{n}}{a_{n} - a_{i}}(a_{k} - a_{n}) \text{ if } k \in F_{2}\\ 0 \text{ if } k = n\end{cases}\]
for $F_{1} = \{k: \varepsilon(k)k \leq i\}$ and $F_{2} = \{k: \varepsilon(k)k \geq i\}$.

\item[(c)] We have $MPP_{\varphi}(\diamond^{n})$ is combinatorially equivalent to the cubical complex formed by gluing together all unit cubes of dimension $\leq n -2$ with vertices contained in $\{\pm 1, 0\}^{n-1} \setminus \{\mathbf{0}\}$. The face lattice of $MPP_{\varphi}(\diamond^{n})$ is isomorphic to the lattice of intervals in the sign poset $\{0, +, -\}^{n-1} \setminus \{\mathbf{0}\}$ ordered under inclusion. 

\item[(d)] Furthermore, $MPP_{\varphi}(\diamond^{n})$ is combinatorially equivalent to $(C_{n-1} + \diamond^{n-1})^{\Delta}$, where $C_{n} = [-1,1]^{n}$ is the $n$-dimensional regular cube.

\item[(e)] The $f$-vector of $MPP_{\varphi}(\diamond^{n})$ is given by \[f_{m}(MPP_{\varphi}(\diamond^{n})) = \sum_{k=1}^{n-m-1} \binom{n-1}{k,m, n-k-m-1} 2^{k+m}. \]
Hence, $MPP_{\varphi}(\diamond^{n})$ has precisely $3^{n-1} -1$ vertices. In particular, they correspond to a sign vector in $\{0, +, -\}^{n-1} \setminus \mathbf{0}$.

\item[(f)] Two vertices in $MPP_{\varphi}(\diamond^{n})$ are adjacent if and only if their corresponding vectors are distance $1$ from one another in the Taxi Cab metric. As a result, $\text{diam}(MPP_{\varphi}(\diamond^{n})) = 2(n-1)  = (n-1)\text{diam}(\diamond^{n}).$

\item[(g)] The total number of monotone paths in $\diamond^{n}$ is precisely $\frac{2^{2n-1} - 2}{3}.$ Not all paths are coherent. The diameter of the entire flip graph of $\diamond^{n}$ is $2(n-1)$, and the longest flip distance to the nearest coherent path is $n-2$.
\end{enumerate}
\end{theorem}

The combinatorial types of these MPPs correspond exactly to the poset of intervals of the sign poset from oriented matroid theory as one would find in Chapter 7 of \cite{zieg}. For this reason, we call polytopes of this combinatorial type \emph{signohedra.} One may view this result as a type $B$ analog to the case of simplices. Namely, the MPPs of simplices are cubes. The face lattice of a cube corresponds to the lattice of intervals of subsets of $[n]$. The type $B$ analog of the simplex is the cross-polytope, and the type $B$ analog of subsets of $[n]$ is the sign poset. Then we may view the signohedron as a type $B$ cube. 

Furthermore, via a functorial lemma proven in \cite{BSFiberPoly}, projections take MPPs to MPPs. The projections of the cross-polytopes are precisely the \emph{centrally symmetric polytopes} or equivalently the set of polyhedral unit balls in $\R^{d}$. Thus, Theorem \ref{MainTheorem1} yields the following corollary: 

\begin{cor}
\label{cor:MainCS}
Let $P$ be a centrally symmetric polytope with $2n$ vertices $\pm v_{1}, \dots, \pm v_{n}$ and linear functional $\ell$ such that $0 < \ell(v_{1}) < \dots < \ell(v_{n})$. Let $a_{i} = \ell(v_{i})$. Then we have the following:
\begin{align*}
    MPP_{\ell}(P) &= \text{conv}\Bigg{(}\Bigg{\{}\left(1 - \frac{a_{i_{k}} + a_{i_{1}}}{2a_{n}}\right)(v_{n}) + \sum_{i = 1}^{k} \left(\frac{a_{i_{k-1}} + a_{i_{k+1}}}{2a_{n}}\right) v_{i_{k}}: \\
    &-n = i_{0} < \dots < i_{k+1} = n \text{ and } i_{a} \neq -i_{b} \text{ for all } a,b \in [k]\Bigg{\}}\Bigg{)}.
\end{align*}
\end{cor}

Note that a similar projection result may be found for all polytopes from projections of simplices and for all zonotopes from projections of cubes. The case for projections of cross-polytopes is more interesting in the sense that some monotone paths are incoherent, and no projection of an incoherent path may be coherent. Our result thus tells us that the longest coherent monotone path on a centrally symmetric polytope with $2n$ is vertices is at most $n$. Therefore, there cannot exist a centrally symmetric equivalent to the Goldfarb cube from \cite{DefProds} in which a coherent monotone path uses all of the vertices of the polytope. 

Understanding the structure of monotone paths on centrally symmetric polytopes could yield insight into the polynomial Hirsch conjecture (see \cite{Hirsch}). That conjecture asks for a polynomial bound on the lengths of paths on polytopes in terms of the number of facets and dimension. This conjecture is of fundamental interest in applications due to its relationship to the run-time of the simplex method for linear programming. The same question for shortest coherent monotone paths also remains open and is connected to the study of the shadow vertex pivot rule studied in depth in \cite{borgwardt} and more recently in \cite{smoothedanalysis}.



\section{Background}
\label{sec:bg} 
Throughout this paper, we rely on a familiarity with convex polytopes at the level of \cite{zieg}. A comprehensive reference for the structure of MPPs may be found in \cite{Baues}, but this section will be sufficient to understand our results.

\begin{defn}
A \emph{monotone path polytope} (MPP) of a polytope $P$ and orientation induced by a linear functional $\ell: P \to \R$ is the fiber polytope induced by the map $\ell: P \to \ell(P)$. In particular, the monotone path polytope is given by 
\[MPP_{\ell}(P) = \text{conv}\left(\left\{\int_{\ell(P)} s(x) dx: s \text{ is a section of }\ell  \right\}\right). \]
\end{defn}

Furthermore, Billera and Sturmfels showed in the same paper that the integrals of the sections of monotone paths generate the monotone path polytope. This observation yields a finite generating set for that infinite space. From this result, we obtain a method to compute monotone path polytopes.

\begin{theorem*}[Restatement of Theorem 5.3 from \cite{BSFiberPoly}]
\label{thm:mppgenerator}
For a linear functional $\varphi$ and polytope $P$ with vertices ordered by the linear functional $p_{1}, p_{2}, \dots, p_{n}$, let $M$ be the set of subsets $S$ of $[n]$ such that $\{p_{s}: s \in S\}$ is a monotone path. Then we have  
\[MPP_{\varphi}(P) = \text{conv}\left(\left\{\sum_{j=2}^{|S|} \frac{\varphi(p_{i_{j}} - p_{i_{j-1}})}{2 \varphi(p_{n} - p_{1})}(p_{i_{j-1}} + p_{i_{j}}): i_{j} \in S, S \in M \right\} \right).\]
\end{theorem*}

The monotone paths that give rise to vertices in the resulting monotone path polytope are called coherent. Faces of the monotone path polytope correspond to coherent cellular strings, a generalization of coherent paths.

\begin{defn}[Page $13$ of \cite{Baues}]
\label{def:cohstring}
Fix a $d-$polytope $P$ and orientation $\ell \in (\R^{d})^{\ast}$. A \emph{cellular string} is a sequence of faces $F_{1}, F_{2}, \dots, F_{n}$ that satisfies the following:
\begin{enumerate}
    \item[(i)] $v_{min} \in F_{1}$ and $v_{max} \in F_{n}$, where $v_{\text{min}}$ and $v_{\text{max}}$ are minimal and maximal vertices respectively with respect to $\ell$.
    \item[(ii)] $\ell$ is non-constant on any $F_{i}$.
    \item[(iii)] For each $i$, the $\ell-$maximizing face of $F_{i}$ is the $\ell-$minimizing face of $F_{i+1}$.
\end{enumerate}
A cellular string is called \emph{coherent} if there exists some linear functional $\ell' \in (\R^{d})^{\ast}$ such that $\bigcup_{i=1}^{n} F_{i}$ is the union of all points $x \in \ell(P)$ of the $\ell'-$minimal points in the fibers $\ell^{-1}(x)$. More concretely, a cellular string is coherent if there is a projection of the whole polytope taking the cellular string to the lower edges of a polygon such that the composition of this projection to the polygon and the map casting a shadow from the polygon given by $\ell.$
\end{defn}

Using the tools we know about fiber polytopes, we may completely describe the face lattice of monotone path polytopes by identifying each face with a coherent cellular string. As an immediate consequence of Theorem $2.1$ in \cite{BSFiberPoly}, the face lattice of a monotone path polytope is equivalent to the lattice of coherent cellular strings with the partial order induced by the refinement of subdivisions.

 We will use this connection to give a complete description of the MPPs of cross-polytopes up to combinatorial equivalence. Applying this result, the coherent monotone paths may be mapped to the lower vertices of some two dimensional projection. Knowing this property yields a simple geometric obstruction to coherence captured by Figure \ref{fig:incoherentpath}. Namely, for any polytope $P$, any coherent monotone path $v_{1}, v_{2}, \dots, v_{n}$ on $P$ must satisfy \[\text{conv}(v_{i}, v_{j}) \cap \text{conv}(\bigcup_{k=i+1}^{j-1} v_{k}) = \emptyset,\]
since the ordered lower vertices of a polygon must satisfy this condition. From this observation, we obtain a general result for coherent monotone paths on centrally symmetric polytopes. Namely, a coherent monotone path on a centrally symmetric polytope cannot contain a pair of antipodes other than its min and max, because convex hulls of distinct pairs of antipodes must intersect as in Figure \ref{fig:incoherentpath}.

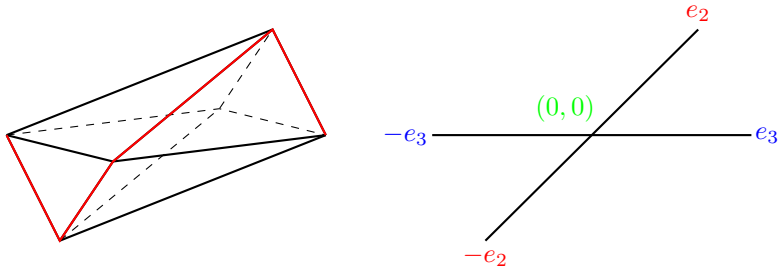
\begin{figure}
    \centering
    \begin{tikzpicture}[scale = .7]
    \draw[black, thick] (-3,0) -- (2,2) -- (3,0);
    \draw[black, dashed] (-3,0) -- (1,.5) -- (3,0);
    \draw[black, thick] (-3,0) -- (-2,-2) -- (3,0);
    \draw[black, thick] (-3,0) -- (-1,-.5) -- (3,0);
    \draw[black, thick] (-2,-2) -- (-1,-.5) -- (2,2);
    \draw[black, dashed] (-2,-2) -- (1,.5) -- (2,2);
    \draw[red, thick] (-3,0) -- (-2,-2) -- (-1, -.5) -- (2,2) -- (3,0);
    \draw[black, thick] (5,0) -- (11,0);
    \draw[black, thick] (6,-2) -- (10,2);
    \draw (4.5, 0) node[blue] {$-e_{3}$};
    \draw (11.3, 0) node[blue] {$e_{3}$};
    \draw (10,2.3) node[red] {$e_{2}$};
    \draw (6,-2.3) node[red] {$-e_{2}$};
    \draw (7.5,.5) node[green] {$(0,0)$};
    \end{tikzpicture}
    \caption{The left figure shows an example of an incoherent path on the octahedron. The obstruction is pictured in the right, since the path contains $-e_{3}, -e_{2}, e_{2},$ and $e_{3}$.  }
    \label{fig:incoherentpath}
\end{figure}

This intuitive geometric obstruction for any centrally symmetric polytope turns out to be the only obstruction to coherence of monotone paths on the cross-polytopes. We will generalize this observation to cellular strings in the next section. The last general fact we require is the following functorial lemma from \cite{BSFiberPoly} that allows for the computation of the monotone path polytope of a projection of a polytope.

\begin{lemma*}[Lemma $2.3$ from \cite{BSFiberPoly}]
\label{ProjLemma}
	Let $P \xrightarrow{\theta} Q \xrightarrow{\varphi} R$ be a sequence of surjective affine maps of polytopes. Then $\Sigma(Q,R) = \theta(\Sigma(P,R))$, where $\Sigma(A,B)$ denotes the fiber polytope for a projection from $A$ to $B$ for polytopes $A$ and $B$. In particular, when $\varphi$ is a linear functional, we find that $MPP_{\varphi}(Q) = \theta(MPP_{\varphi \circ \theta}(P)).$
\end{lemma*}

This lemma allows for the computation of monotone path polytopes of any centrally symmetric polytope as the projection of a signohedron and makes Corollary \ref{cor:MainCS} immediate from the proof of Theorem \ref{MainTheorem1}(a). 

\section{Signohedra: Monotone Paths on Cross-Polytopes}
\label{sec:signo}

In this section, we completely describe signohedra and equivalently the monotone path polytopes on cross-polytopes via the proof of Theorem \ref{MainTheorem1}. To start studying the MPPs of cross-polytopes for a generic orientation, we must clarify what constitutes a generic orientation. Using this notion of generic, we will fix an ordering of the vertices of the cross-polytope that will be used for all remaining computations of the MPP.

\begin{lemma}
\label{lem:cpgenorient}
Generically, a monotone path polytope of $\diamond^{n}$ is affinely equivalent to one obtained from a linear functional with distinct positive values for each $e_{i}$.
\end{lemma}

\begin{proof}
The vertex generic linear functionals on the cross-polytope are precisely those with distinct nonzero values of coefficients $a_{i}$ for each $e_{i}$. Furthermore, they each map $e_{i}$ to $a_{i}$, where up to a change in indices, $|a_{1}| < |a_{2}| < \dots < |a_{n}|$. Then, by applying the reflection map taking $e_{i} \mapsto -e_{i}$, we may assume that each $a_{i}$ is positive. The cross-polytope $\diamond^{n}$ has vertices $\pm e_{i}$, so under this map, the vertices are ordered such that $-e_{i} < e_{i}$ for all $i \in [n]$ and $e_{j} < e_{k}$ for all $j < k$ in $[n]$. Hence, up to a permutation and reflection, we always obtain the same vertex ordering. Since these symmetries are linear, by Lemma $2.3$ from \cite{BSFiberPoly}, the affine isomorphism from the cross-polytope to itself induces an affine isomorphism of the monotone path polytopes.
\end{proof}

For cubes and simplices, all monotone paths are coherent. This property makes understanding their monotone path polytopes easier. For cross-polytopes with an orientation given by a generic linear functional, the monotone paths need not all be coherent as noted in Section \ref{sec:bg}. The following theorem is the primary technical fact from which all of our remaining work on the characterization follows. 

\begin{theorem}
\label{CoherenceThm}
A cellular string on $\diamond^{n}$ is coherent if and only if the set of vertices that appear in the string only contains one pair of antipodes, namely the maximum and minimum pair.
\end{theorem}

\begin{proof}
For both directions, by Lemma \ref{lem:cpgenorient}, we may assume without loss of generality that the linear functional $\ell: \R^{n} \to \R$ given by $\ell(e_{i}) = a_{i}$ satisfies $0 < a_{i} < a_{j}$ for all $i, j \in [n]$ with $i <j$. 

Suppose that a coherent cellular string contains $-e_{i}$ and $e_{i}$ as vertices in possibly distinct cells. Then, by the definition of coherence, there exists a projection $\pi: \diamond^{n} \to \R^{2}$ that takes $-e_{i}$ and $e_{i}$ to possibly distinct lower edges of some polygon, where $\pi = \ell \times \varphi$ for some linear functional $\varphi: \R^{n} \to \R$. Since $\pi(-e_{i})$ and $\pi(e_{i})$ lie on the lower edges, and $-e_{n}$ and $e_{n}$ are minimal and maximal respectively for $\varphi$, $\pi(e_{i})$ and $\pi(-e_{i})$ must lie below the line segment from $\pi(-e_{n})$ to $\pi(e_{n})$. It follows that the slope from $\pi(-e_{n})$ to $\pi(-e_{i})$ must be less than the slope from $\pi(-e_{n})$ to $\pi(e_{n})$. Similarly, the slope from $\pi(e_{i})$ to $\pi(e_{n})$ must be greater than the slope from $\pi(-e_{n})$ to $\pi(e_{n})$. Thus, we must have 
\[\frac{\varphi(e_{n}) - \varphi(e_{i})}{a_{n} - a_{i}} = \frac{\varphi(-e_{i}) - \varphi(-e_{n})}{-a_{i} + a_{n}} < \frac{\varphi(e_{n}) - \varphi(-e_{n})}{a_{n} + a_{n}}< \frac{\varphi(e_{n}) - \varphi(e_{i})}{a_{n} - a_{i}}, \]
a contradiction. 

Suppose instead we have a cellular string such that the set of vertices contained in some cell has only one pair of antipodes. The vertices contained in the cellular string may be partitioned into the subsets of positive and negative basis vectors: $S_{+} \subseteq \{e_{1}, e_{2}, \dots, e_{n}\}$ and $S_{-} \subseteq \{-e_{1}, -e_{2}, \dots, -e_{n}\}$. Let $S_{0} = \{e_{i}: -e_{i}, e_{i} \notin S_{+} \cup S_{-}\}$, the set of $e_{i}$ such that $\pm e_{i}$ does not appear in the path. Then, by our assumption that the path only contains a single pair of antipodes $-e_{n}$ and $e_{n}$, we must have 
\[S_{+} \cap -S_{-} = \{e_{n}\}.\]
It follows that $(S_{+} \cup -S_{-} \cup S_{0}) \setminus \{-e_{n}\} = \{e_{1}, e_{2}, \dots, e_{n}\}$ and is, in particular, linearly independent. Let $F_{1} < F_{2} < \dots < F_{k}$ denote the sequence of faces in the cellular string. Let $e_{-i} = -e_{i}$ and $a_{-i} = -a_{i}$ for $i \in \pm [n]$. To prove coherence, we must choose $\varphi: \R^{n} \to \R$ such that $\pi = \ell \times \varphi$ takes endpoints of the cellular string to lower vertices of $\pi(\diamond^{n})$ and vertices of cells to the interior of the edge between the endpoints. We will construct such a choice of $\varphi$ inductively first on the endpoints of the string and then interpolate to find the value on the interior points. 

By linear independence, we may define $\varphi$ however we choose for each vertex that appears in the cellular string. Let $e_{b_{j}}$ and $e_{c_{j}}$ denote the minimal and maximal vertices of $F_{j}$. Define $\varphi(e_{n}) = 0$. Define $\varphi(e_{c_{1}})$ to be $-(a_{c_{1}}+a_{n})$. Then the slope from $(-a_{n}, \varphi(-e_{n}))$ to $(a_{c_{1}}, \varphi(a_{c_{1}}))$ will be precisely $-1$. Define $\varphi$ inductively so that the slope from $(a_{b_{j}}, \varphi(e_{b_{j}}))$ to $(a_{c_{j}}, \varphi(e_{c_{j}}))$ is $\frac{-1}{j}$ for all $1 \leq j < k$. For each remaining vertex $v$ in each $F_{j}$, define $\varphi$ so that $\pi(v)$ lies on the line segment from $\pi(e_{b_{j}})$ to $\pi(e_{c_{j}})$. Such a choice is always possible by linear independence. Finally, define $\varphi$ to be $0$ for all vertices in $S_{0}$. 

 It remains to show that $\pi$ then satisfies the properties from Definition \ref{def:cohstring}. Observe that, since $\varphi(-e_{n}) = -\varphi(e_{n}) = 0$ and the slope between consecutive endpoints of the string is negative, $\varphi$ is negative for each endpoint of the cellular string other than $-e_{n}$ and $e_{n}$. Thus, by interpolating, $\varphi$ must be negative for each vertex on the interior of a cell. For vertices $e_{i}$ not in the string, there are two cases. If $-e_{i}$ is in the string, then $\varphi(e_{i}) = - \varphi(-e_{i}) \geq 0$. Otherwise, $e_{i} \in S_{0}$, so $\varphi(e_{i}) = 0 \geq 0$. Hence, all vertices $v$ not contained in some cell of the string must satisfy $\varphi(v) \geq 0$. Since $\varphi(-e_{n}) = \varphi(e_{n}) =0$, it follows that all vertices not in the string lie on or above the line segment from $\pi(-e_{n})$ to $\pi(e_{n})$. 
 
 Thus, the lower vertices of the polygon must be some subset of the vertices contained in the string. By construction, we defined the $F_{i}$ to each be mapped to an edge and so that the slope of each edge increases as $i$ increases. These edges yield a path from $\pi(-e_{n})$ to $\pi(e_{n})$. Since the slope is increasing, this path is the graph of a piece-wise linear convex function that lies below the line segment from $\pi(-e_{n})$ to $\pi(e_{n})$. It follows then that the convex hull of the path has vertices $\{\pi(e_{i}): e_{i} \text{ is an endpoint of the string}\}$. Furthermore, by construction once again, any vertex in a cell of the cellular string is mapped to the interior of the edge between the endpoints of that string. Hence, the projections of each $F_{i}$ are precisely the lower edges of the polygon meaning that the cellular string must be coherent by definition.
\end{proof}

\begin{figure}
    \centering
    \begin{tikzpicture}[scale = .8]
  \draw[black, thick] (0,0) -- (1,-1) -- (2,-1.8)-- (3, -2.4) -- (4,-2.8) -- (5, -3) -- (6, 0);
  \draw[black, thick] (0,0) -- (1,3) -- (2,2.8)-- (3, 2.4) -- (4,1.8) -- (5, 1) -- (6, 0);
  \draw (-.5,0) node[red] {$-x_{n}$}; 
  \draw (1,-1.25) node[red] {$x_{i_{1}}$};
  \draw (2,-2) node[red] {$x_{i_{2}}$};
  \draw (3, -2.65) node[red] {$x_{i_{3}}$};
  \draw (4,-3.05) node[red] {$x_{i_{4}}$}; 
  \draw (5, -3.25) node[red] {$x_{i_{5}}$}; 
  \draw (6.25,0) node[red] {$x_{n}$};
  \draw (1,3.25) node[blue] {$-x_{i_{5}}$};
  \draw (2,3) node[blue] {$-x_{i_{4}}$};
  \draw (3,2.65) node[blue] {$-x_{i_{3}}$};
  \draw (4,2) node[blue] {$-x_{i_{2}}$}; 
  \draw (5, 1.25) node[blue] {$x_{i_{1}}$}; 
  \draw (.5,0) node[green] {$x_{s_{1}}$};
  \draw (1.5,0) node[green] {$x_{s_{2}}$}; 
  \draw (4.5,0) node[green] {$-x_{s_{2}}$};
  \draw (5.5,0) node[green] {$-x_{s_{1}}$};
  \draw (2.5,0) node[green] {$x_{s_{3}}$};
  \draw (3.5,0) node[green] {$-x_{s_{3}}$};
\end{tikzpicture}
    \caption{The proof of Theorem \ref{CoherenceThm} may be visualized. Namely, the red $x_{i_{j}}$ in the image represent the endpoints of each cell in the string. Via linear independence, we may impose that the slopes between these vertices are increasing and that they all lie below the line between $-x_{n}$ and $x_{n}$. Furthermore, we may impose that all vertices in each cell are mapped to the interior of an edge via linear interpolation and linear independence. Then their antipodes, the blue $-x_{i_{j}}$, must lie above that line from $-x_{n}$ to $x_{n}$. For the remaining vertices, the green $x_{s_{i}}$, we may again, by linear independence, place them all on the line between $-x_{n}$ and $x_{n}$, which forces their antipodes to also lie on that line. }
    \label{fig:mainproof}
\end{figure}
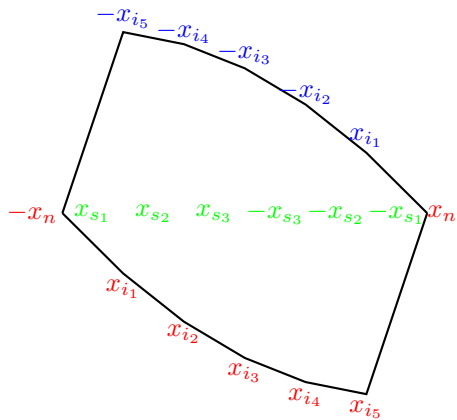

Interpreted for cellular strings corresponding to monotone paths, we established what was suggested in Section \ref{sec:bg}. Namely, a monotone path on $\diamond^{n}$ is coherent if and only if the only antipodes it contains are the maximum and minimum pair. Note the assumption that the linear functional is generic is necessary here.

\begin{prop}
For a cross-polytope and the orientation $\ell = \sum_{i=1}^{n} e_{i}^{T}$, the monotone path polytope is given by $\Delta_{d-1} -\Delta_{d-1}$. 
\end{prop}

\begin{proof}
In \cite{BSFiberPoly}, they show that the fiber polytope is given by the minkowski sum of fibers of barycenters of the subdivision induced by the projection. In this case, the subdivision induced by the projection is trivial, so the monotone path polytope is given by $\ell^{-1}(0)$. In that case, we are taking a slice of the Cayley sum of $\Delta_{d-1}$ and $-\Delta_{d-1}$, so from either explicit computation or the Cayley trick as in \cite{Cayley}, we find that the resulting MPP is $\Delta_{d-1} - \Delta_{d-1}$. 
\end{proof}

All monotone paths are given by single edges, so it is immediate that they are all coherent unlike in the generic case. As an immediate corollary of this:
\begin{cor}
All monotone paths being coherent for one orientation does not imply that all monotone paths are coherent for all orientations. Furthermore, a polytope may not have all paths coherent for any generic orientation but still have some orientation for which all monotone paths are coherent.
\end{cor}

As stated in Section \ref{sec:bg}, by Theorem $2.1$ of \cite{BSFiberPoly}, the coherent paths correspond exactly to the vertices of the monotone path polytope. From this result, we may immediately compute the number of vertices.

\begin{cor}
\label{cor:seqbij}
For a generic linear functional $\varphi$, $MPP_{\varphi}(\diamond^{n})$ has precisely $3^{n-1} -1$ vertices. In particular, they correspond to elements of $\{-1, 1, 0\}^{n-1} \setminus \{\mathbf{0}\}$.
\end{cor}

\begin{proof}
Recall that any two non-antipodal points of the cross-polytope are connected by an edge. It follows then that the coherent monotone paths consists of choices $e_{i}, -e_{i}$ or neither to include in our sequence of points. That gives $3^{n-1}$ possible choices. Since we have to include at least $1$ element between $-e_{n}$ and $e_{n}$, we obtain that $\diamond^{n}$ has precisely $3^{n-1}-1$ coherent monotone paths. Since the vertices of $MPP_{\varphi}(\diamond^{n})$ correspond to coherent monotone paths, $MPP_{\varphi}(\diamond^{n})$ has precisely $3^{n-1}-1$ vertices.
\end{proof}

We may strengthen this result to find explicit vertices via computing the average value of each coherent monotone path.

\begin{proof}[Proof of Theorem \ref{MainTheorem1}(a)]
Use the characterization of coherent monotone paths in Theorem \ref{CoherenceThm} in combination with Theorem $5.3$ from \cite{BSFiberPoly}, and the result is immediate.
\end{proof}

At this point, we may prove Corollary \ref{cor:MainCS}.

\begin{proof}[Proof of Corollary \ref{cor:MainCS}]
Apply Lemma $2.3$ from \cite{BSFiberPoly} as stated in Section \ref{sec:bg} to the result of Theorem \ref{MainTheorem1}(a).
\end{proof}

While more involved, we may similarly provide an explicit facets based description of the polytope from our proof of Theorem \ref{CoherenceThm}. 

\begin{proof}[Proof of Theorem \ref{MainTheorem1}(b)]
To find the facet defining relations for the MPP, we follow the method outlined by Ziegler in Section $9.1$ of \cite{zieg}. Namely, the facet defining inequalities are obtained from the linear functional that yields the polygon whose lower vertices correspond to maximal coherent subdivisions. That is, for such a linear functional $\varphi$, the inequality is given by $\varphi(x) \geq \varphi(v)$, where $v$ is the vertex that is minimized by $\varphi$ on the polygon corresponding to the maximal coherent cellular string. From the combinatorial characterization in Theorem \ref{MainTheorem1}(b), the facets correspond precisely to the maximal intervals in the sign poset. Maximal intervals are obtained from starting with a choice of a separating vertex $e_{i}$ and choosing a maximal length monotone path through $e_{i}$. In the notation from the proof of Theorem \ref{CoherenceThm}, these are precisely the subdivisions corresponding $F_{1} < F_{2}$ such that $\pm F_{1} \cup \pm F_{2} = \{k:k \in \pm [n-1]\}$, where we remove $e_{n}$ from $F_{2}$ and $-e_{n}$ from $F_{1}$ and identify each $e_{i}$ with $i$. To obtain a lifting functional, by following the proof of Theorem \ref{CoherenceThm}, we define \[\varphi(e_{n}) = 0 \text{ and } \varphi(e_{i}) = -a_{i} - a_{n}.\]
Then for the remaining vertices in $F_{1}$ we linearly interpolate between $0$ and $-a_{i} - a_{n}$. For the vertices in $F_{2}$, we provide a similar interpolation. In particular, $\varphi(e_{k}) = -a_{k} - a_{n}$ if $k \in F_{1}$ and $\varphi(e_{k}) = \frac{a_{i} + a_{n}}{a_{n} - a_{i}}(a_{k} - a_{n})$ if $k \in F_{2}$
\[\varphi(e_{k}) = \begin{cases}  -a_{k} - a_{n} \text{ if } k \in F_{1} \\   \frac{a_{i} + a_{n}}{a_{n} - a_{i}}(a_{k} - a_{n}) \text{ if } k \in F_{2}\\ 0 \text{ if } k = n\end{cases}.\]
We denote any functional of this form as $\varphi_{i, \varepsilon}$, where $i \in \pm [n-1]$ is the choice of the splitting vertex, and $\varepsilon: [n-1] \to \pm 1$ denotes the sign sequence of the vertices. Then $F_{1} = \{k: \varepsilon(k)k \leq i\}$ and $F_{2} = \{k: \varepsilon(k)k \geq i\}$, which yields the result.
\end{proof}

Thus, now we have an explicit description of our polytope in terms of both vertices and facets. We may take this description a step further and use our characterization of coherent cellular strings to obtain a complete characterization of the face lattice of the MPP of $\diamond^{n}$ in terms of the sign poset on $\{+, - ,0\}^{n-1} \setminus \mathbf{0}$. 
\begin{figure}
    \centering
    \includegraphics[scale = .25]{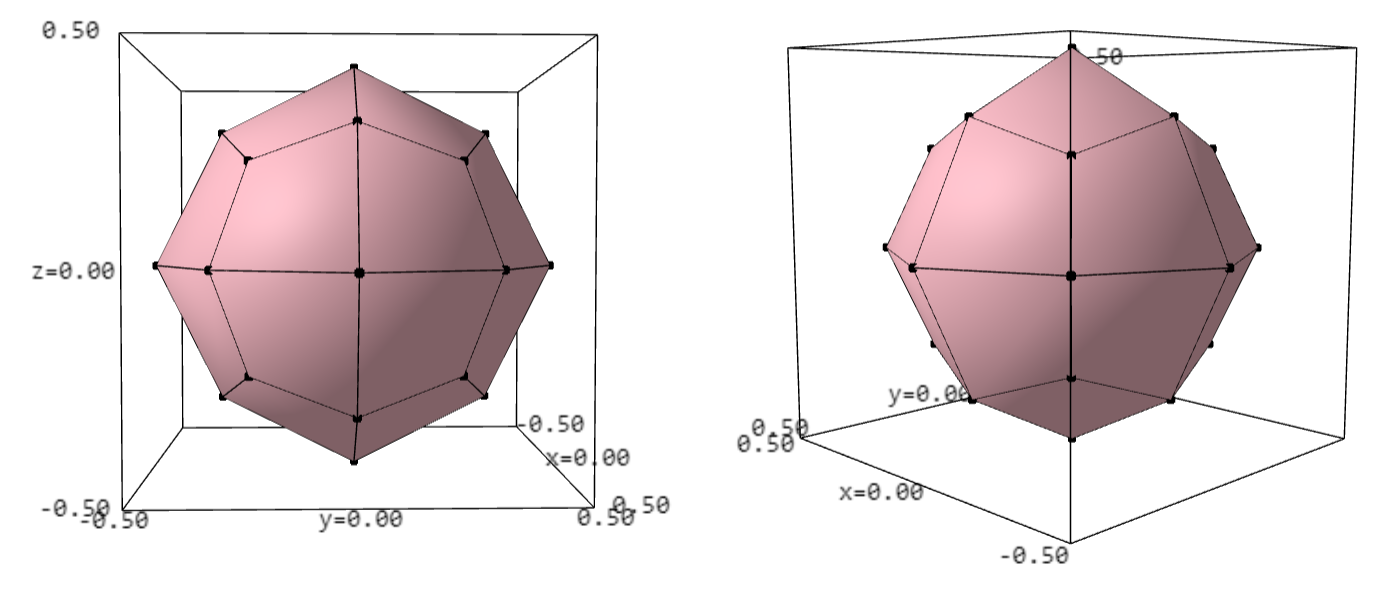}
    \caption{A plot of $(C_{3} + \diamond^{3})^{\Delta}$ made using \cite{sage}. By Theorem \ref{MainTheorem1}(c), any MPP of $\diamond^{4}$ for a generic orientation is combinatorially equivalent to the pictured polytope.}
    \label{fig:cpmpp}
\end{figure}

\begin{proof}[Proof of Theorem \ref{MainTheorem1}(c)]
By Theorem $2.1$ of \cite{BSFiberPoly}, the face lattice of the monotone path polytope corresponds to the poset of coherent cellular strings under refinement. Note that a coherent cellular string is uniquely determined by two pieces of data: its endpoints and the vertices included in the cells between each endpoint. These two pieces of data may be interpreted as two monotone paths, the one with vertices given by the set of endpoints, $E$, of the cellular string and the on given by $V$, the set of vertices that appear anywhere in the cellular string. Clearly $E \subseteq V$, which translates via the bijection to $V$ and $E$ being comparable. Furthermore, the vertices contained in the cellular string correspond exactly to all monotone paths with a set of vertices $S$ such that $E \subseteq S \subseteq V$. Hence, the partial order of inclusion of faces via this bijection is equivalent to the partial order given by inclusion of intervals. 

Note that each interval $I = [a, b]$ in $\{0,-1,1\}^{n-1} \setminus \{\mathbf{0}\}$ is Boolean. Furthermore, $\text{conv}([a,b])$ is isomorphic to $\text{conv}([a,b] - a) = \text{conv}([\mathbf{0}, b-a])$, where $\mathbf{0}$ is adjoined in the natural way to the partial order. Let $k$ be the length of the interval. Then $b-a$ will have precisely $k$ nonzero entries of either $0$'s or $1'$s. Let $A$ be the linear map defined by $A(e_{i}) = \sigma(i)e_{i}$, where $\sigma(i)$ denotes the sign of $e_{i}$ in $b-a$. Let $S = \{i \in [n-1]: \sigma(i) \neq 0\}$. Then $A([0,b-a]) = \left[0, \sum_{s \in S} e_{s} \right]$. By construction, $A$ is then an isometry taking $\text{conv}(I)$ to the unit cube. 

Hence, each interval corresponds to a unit cube of dimension $\leq n-2$ with vertices contained in $\{\pm 1, 0\}^{n-1} \setminus \{\mathbf{0}\}$. The converse is similar, since each unit cube has vertices corresponding to an interval in the poset $\{+,-,0\}^{n-1} \setminus \{\mathbf{0}\}$
\end{proof}

 The next step is to show that this combinatorial type has a nice representative given by $(C_{n-1} + \diamond^{n-1})^{\Delta}$.

\begin{proof}[Proof of Theorem \ref{MainTheorem1}(d)]
 Note that the vertices of $C_{n} + \diamond^{n}$ are obtained precisely by adding the unique maximal vertices on each for a linear functional. Let $\ell$ be a linear functional. Then the max on $\diamond^{n}$ is the vertex corresponding to the coordinate of maximum absolute value together with the corresponding sign, and the max on $C_{n}$ returns the subset with $1$'s for positive elements and $-1$'s for negative elements. The element of maximum absolute value could either be positive or negative.

The resulting polytope has vertices $\{B_{d}\left(2e_{1} + \sum_{i=2}^{d} e_{i}\right)\}$, where $B_{d}$ is the set of signed permutation matrices. Let $S \in \{+, -,0\}^{n} \setminus \{\mathbf{0}\}$. Then $S$ induces a partitions of $[n]$ into $S_{+} \cup S_{-} \cup S_{0}$, and there is a naturally associated linear functional $\varphi_{S}$ given by $\sum_{i \in S_{+}} e_{i}^{T} - \sum_{j \in S_{-}} e_{j}^{T}.$ The vertices maximized by $\varphi_{S}$ are precisely 
\[\{\text{sign}(k)e_{k}: k \in S_{+} \cup S_{-}\} + \{\sum_{a \in A} \varepsilon(a) e_{a}: A \in S_{0}, \varepsilon: A \to \{\pm 1\}\} + \sum_{k \in S_{+} \cup S_{-}} \text{sign}(k)e_{k}. \]
These vectors span the affine hyperplane given by 
\[\sum_{i \in S_{+}} x_{i} - \sum_{j \in S_{-}} x_{j} = |S_{+}| + |S_{-}| + 1.\]
Thus, each sign vector $\varphi_{S}$ corresponds to a facet of the resulting polytope. Let $\varphi$ be a linear functional. Then any vertex maximized by that linear functional would also be maximized by the linear functional with the same sign pattern. Hence, the sign vectors induce all of the facets of $C_{n} + \diamond^{n}$, which gives us a polyhedral formulation of this polytope.

Namely, for any partition $S_{+} \cup S_{-} \cup S_{0}$ of $n$, we must have
\[\frac{1}{|S_{+}| + |S_{-}| + 1}\left(\sum_{i \in S_{+}} x_{i} - \sum_{j \in S_{-}} x_{j}\right) \leq 1.\]
These are precisely the relations given by maximizing the sign functionals. Observe that if a sign vector contains $0$'s, then any choice of $\pm e_{i}$ for $i \in S_{0}$ is allowable for a vertex in that facet. Furthermore, if two facets have distinct positive sets or negative sets, then they cannot intersect, since that imposes the sign of $e_{i}$ for any vertex in the set. It follows that two facets intersect if and only if their corresponding sign vectors are comparable. In particular, $m$-dimensional faces correspond exactly to intervals of sign vectors in the poset of length $n-m$. It follows then that the face lattice of this polytope is isomorphic to the lattice of intervals of the sign poset under reverse inclusion. Hence, we have $(C_{n-1} + \diamond^{n-1})^{\Delta}$ and $\text{MPP}_{\varphi}(\diamond^{n})$ are combinatorially equivalent.
\end{proof}
 An advantage of this representation is that we may read off each vertex for a sign vector as
\[\frac{1}{|S_{+}| + |S_{-}|+1} \left(\sum_{i \in S_{+}} e_{i} - \sum_{j \in S_{-}} e_{j} \right). \]
One may appeal to the theory of anti-prisms for an alternative proof of Theorem \ref{MainTheorem1}(d). Namely, for perfectly centered polytopes, Bj\"{o}rner showed in \cite{bjornerantiprism} that the lattice of intervals in the face lattice of a polytope $P$ under inclusion is isomorphic to the face lattice of $(P + P^{\Delta})^{\Delta}$. Alongside this result from the characterization of the face lattice, we also now have a combinatorial framework for calculating the $f$-vector.

\begin{proof}[Proof of Theorem \ref{MainTheorem1}(e)]
From Theorem \ref{MainTheorem1}(b), faces correspond exactly to elements of the poset of intervals in the sign poset. Namely, they are identified precisely by pairs of elements of comparable elements of the poset. An $m$-face corresponds to pairs of elements of distance $m$ from each other. That is one has $k$ nonzero entries, and the other has $k+m$ nonzero entries including the $k$ nonzero entries of the starting point. Thus, the faces correspond to flags of length two of subsets of $n-1$ of subsets of size $k$ and $k + m$ counted by $\binom{n-1}{k,m, n-k-m-1}$ together with $2^{k+m}$ choices of signs for each vertex contained in the flag. Then we have that

\[f_{m}(MPP_{\pi}(\diamond^{n})) = \sum_{k=1}^{n-m-1} \binom{n-1}{k,m, n-k-m-1} 2^{k+m}. \]
\end{proof}

Again by Theorem $2.1$ of \cite{BSFiberPoly}, edges in the monotone path polytope correspond to refinements of pairs of coherent monotone paths. Geometrically, we may interpret this refinement as two monotone paths agreeing everywhere except on a single two dimensional face. In the sense of the flip graph, this interpretation means that the two monotone paths differ by a polygonal flip. From this observation, we obtain the following lemma. 

\begin{lemma}
Two vertices in the $MPP$ of $\diamond^{n}$ are adjacent if and only if their corresponding vectors are distance one from one another in the Taxi Cab metric. 
\end{lemma}

\begin{proof}
Recall that $\diamond^{n}$ is simplicial. It follows that a polygonal flip either deletes a vertex from a path or adds a single, new vertex. Deleting or adding a vertex corresponds to changing a $1$ or $-1$ to a $0$ or a $0$ to a $1$ or $-1$ in the sequence bijection from Corollary \ref{cor:seqbij}. The connectivity of $\diamond^{n}$ allows us to perform this operation for any element of the sequence. Two sequences of $1$'s, $0$'s, and $-1$'s are at distance $1$ in the Taxi-cab metric if and only if they agree on all but one entry in which one is a $0$ and the other is a $-1$ or $1$, which yields the result.
\end{proof}

Via explicit computation, we may then compute the diameter of the $MPP$ of $\diamond^{n}$:

\begin{proof}[Proof of Theorem \ref{MainTheorem1}(f)]
By the triangle inequality,
\[\sup_{x,y \in \text{Verts}(MPP(\delta^{n}))} |x-y|_{1} \geq |\sum_{i=1}^{n-1} e_{i} - \sum_{i=1}^{n} -e_{i}|_{1} \geq 2(n-1).\]
Since each step along an edge can change the distance by at most one, we have the diameter is at least $2(n-1)$. To see that it is at most $2(n-1)$ may be seen by taking two vectors and changing the coordinates by $\pm 1$ until one vector equals the other. Each coordinate change represents a single step, and the path requires at most $2(n-1)$ coordinate changes. The only detail is avoiding the origin, and that is also easy.
\end{proof}

Now that we understand the graph of $MPP$ of $\diamond^{n}$, to obtain a more complete description of the space of monotone paths, we will describe the flip graph. The flip graph has as its vertices the set of all monotone paths on a polytope with edges given by polygon flips. We then enumerate its vertices and compute its diameter. 

\begin{proof}[Proof of Theorem \ref{MainTheorem1}(g)]
A monotone path corresponds to a subsequence $s_{1}, s_{2}, \dots, s_{m}$ of 
\[(-e_{n-1}, -e_{n-2}, \dots, -e_{1}, e_{1}, e_{2}, \dots, e_{n-1})\]
such that $s_{k} + s_{k+1} \neq 0$, since all vertices are connected to all vertices other than their antipodes. There are $2^{2(n-1)}-1$ non-empty subsets of $\{ e_{i} : i \in \{\pm 1, \pm 2, \dots, \pm n-1\}$. Then $2^{2(n-2)}$ of those subsets include $-e_{1}, e_{1}$, $2^{2(n-3)}$ include $-e_{2}, e_{2}$ but neither of $-e_{1}, e_{1}$, and in general $2^{2(n-1-k)}$ include $-e_{k}, e_{k}$ but none of $e_{j}$, where $|j| < |k|$. Hence, one may easily verify via standard results for geometric series, the resulting number of possible sequences is 
\[2^{2(n-1)}-1 - \sum_{k=1}^{n-1} 2^{2(n-k-1)}  = \frac{2^{2n-1} - 2}{3}.\]

For the diameter of the flip graph, since $\diamond^{n}$ is simplicial, two monotone paths are adjacent if and only if they differ by the addition or removal of a single vertex. A vertex may only be added or removed if it does not introduce consecutive antipodal points. Note that, given these restrictions, the distance between the path given by $e_{1}$ and the path $(-e_{n-1}, -e_{n-2}, \dots, -e_{1}, e_{2}, e_{3}, \dots, e_{n})$ is at least $2(n-1)$. By starting with removing all negative vertices starting with $-e_{n-1}$ and ending with $-e_{1}$ and then adding $e_{1}$ and removing $e_{2}$ through $e_{n}$ we achieve a sequence of flips going between these paths of length precisely $2(n-1)$. Hence, the distance between those points is precisely $2(n-1)$. 

Let $e_{s_{i}}$ and $e_{t_{j}}$ be two different paths. Let $s_{-}$ and $s_{+}$ denote the maximal negative element and minimal positive element of $s$ respectively. define $t_{-}$ and $t_{+}$ similarly. If $s_{-} = t_{-}$ and $s_{+} = t_{+}$, we may go from $s$ to $t$ by adding elements from $t$ that are not in $s$ to $s$ and taking away elements from $s$ that are not in $t$. Such a path will have length at most $2(n-2)$. 

If $s_{-} < t_{-}$ and $s_{+} = t_{+}$, then we may add $t_{-}$ to $s$ and follow the same strategy. This will result in a sequence of moves of length at most $2(n-2) + 1$. A similar idea works for any of the possible cases in which $s_{-} = t_{-}$ or $s_{+} = t_{+}$. 

Suppose that $s_{-} < t_{-}$ and $s_{+} < t_{+}$. If $t_{-} \neq - s_{+}$, we may add $t_{-}$ to the list $s_{-}$. Then we may follow the same strategy for the remaining list keeping $s_{+}$ and $t_{-}$. Then, at the end, we remove $s_{+}$. The result must take fewer that $2(n-2)+2 \leq 2(n-1)$ moves. Suppose instead that $s_{-} < t_{-} = -s_{+} < s_{+} < t_{+}$. First modify all elements of $s$ greater that $s_{+}$ to agree with $t$. If $t_{+} = -s_{-}$, remove $t_{+}$. Otherwise, leave it. In both cases, add $t_{-}$, add $t_{+}$ back and modify all elements $< t_{-}$ to agree with $t$. The result takes $\leq 2(n-2)- 1+3 = 2(n-1)$ moves, since $e_{1} < t_{+}$.

The only remaining case is that $s_{-} < t_{-}$ and $s_{+} > t_{+}$. In that case, add $t_{-}$ and $t_{+}$ to $s$ and make the required changes. The result takes fewer than $2(n-2) + 2$ moves. Hence, the diameter of the flip graph is $2(n-1)$.

The longest distance to the nearest coherent path for a monotone path may be computed similarly. Start with a path $s$ with $s_{-}$ and $s_{+}$ defined as before. If $s_{+} < - s_{-}$, remove all parts of antipodal pairs after $s_{+}$. Otherwise, remove all parts of antipodal pairs before $s_{-}$. Since there are at most $n-2$ elements before $s_{-}$ and $s_{+}$, the distance to the nearest coherent path is at most $n-2$. This bound is attained for the path $(-e_{n-1}, -e_{n-2}, \dots, -e_{1}, e_{2}, e_{3}, \dots, e_{n-1})$. 
\end{proof}
Thus, the total number of monotone paths grows at a rate of $\Theta(4^{n})$, which is exponentially faster than growth rate of the number of coherent monotone paths, which grows at a rate of $\Theta(3^{n})$. That last proof concludes our description of the structure of monotone paths on the cross-polytopes and our proof of Theorem \ref{MainTheorem1}.

\section*{Acknowledgments}  We are grateful to Lionel Pournin, Raman Sanyal, and Bernd Sturmfels for useful comments and support. The authors gratefully acknowledge partial support from NSF DMS-grant 1818969.
\bibliographystyle{amsplain}
\bibliography{References.bib} 


\end{document}